\documentclass[graybox]{svmult}
\usepackage{mathptmx}       
\usepackage{makeidx}         
\usepackage{graphicx}        
\usepackage{multicol}        
\usepackage[bottom]{footmisc}
\usepackage{url}

\usepackage{amsfonts}
\usepackage{amsmath}
\usepackage{amssymb}
\usepackage{wrapfig}
\usepackage{csquotes}
\usepackage{floatflt}
\usepackage{algorithm}
\usepackage{algpseudocode}
\usepackage[font=small,skip=0pt]{caption}

\graphicspath{{./Graphics}}

\begin{document}

\title{Bi-level regularization via iterative mesh refinement for aeroacoustics}
\author{Christian Aarset and Tram Thi Ngoc Nguyen} 
\institute{C. Aarset \at University of G\"ottingen - Institute for Numerical and Applied Mathematics, Lotzestr. 16-18, D-37083 G\"ottingen, Germany  
\and T. T. N. Nguyen \at MPI Solar Systems Research - Fellow Group Inverse Problems, Justus-von-Liebig-Weg 3, 37077 G\"ottingen, Germany \at \email{nguyen@mps.mpg.de}}  

\maketitle

\abstract{In this work, we illustrate the connection between adaptive mesh refinement for finite element discretized PDEs and the recently developed \emph{bi-level regularization algorithm}, due to Nguyen (Inverse Problems \textbf{40}(4) 2024). By adaptive mesh refinement according to data noise, regularization effect and convergence are immediate consequences. We moreover demonstrate its numerical advantages to the classical Landweber algorithm in terms of time and reconstruction quality for the example of the Helmholtz equation in an aeroacoustic setting.}




\def\R{\mathbb{R}}
\def\C{\mathbb{C}}
\def\N{\mathbb{N}}
\def\F{\mathbb{F}}
\def\Om{\Omega}
\def\Ls{L^2(\Omega_0)}
\def\Ld{L^2(\Omega_1,\C)}
\renewcommand{\Re}{\mathop{Re}\nolimits}
\renewcommand{\d}{\,\mathrm{d}}

\newcommand{\blue}[1]{\textcolor{blue}{#1}}
\section{Inverse aeroacoustic source with Helmholtz equation} We investigate an inverse source problem in aeroacoustics, where an unknown source $\phi$  is located in the sub-region $\Omega_0$ of an enclosed room, denoted by $\Omega$, and needs to be determined from acoustic oscillations measured outside of $\Omega_0$. The measurement region, denoted $\Omega_1\subset\Omega$, includes a finite number of obstacles, e.g.~sound-hard walls, reflecting the sound wave propagating to the aeroacoustic sensors \cite{aeroacoustic}. $\Omega$ is treated as a homogeneous medium, enabling the Helmholtz equation to adequately model the oscillation at a single frequency and constant sound speed.

Mathematically, let the room $\Omega$ be a bounded Lipschitz domain in $\R^d$, $d=1,2,3$; for illustration in this article, we choose $d=2$. The unknown real source
\begin{align}
\phi\in \Ls, \quad \Omega_0\subset\text{int }\Omega
\end{align}
generates a complex wave field $u$, propagating throughout the room. The measurement region $\Om_1\subset\Om$ is disjoint from the source domain $\Om_0$ in the sense that $\overline{\Om_0}\cap\overline{\Om_1}=\emptyset$; we here assume that $\Om_1$ additionally contains three disjoint rectangular sound-hard scatterers $S_i$, $i\in\{1,2,3\}$; see Section \ref{sec:num}. We then define the measurement operator with $\Omega_1:=\Omega\setminus\big(\Omega_0\cup\bigcup_{i=1}^3S_i\big)$ as
\begin{equation}\label{data}
\begin{split}
M: &\, H^1(\Omega\setminus\bigcup_{i=1}^3S_i,\C)\to  L^2(\Omega_1,\C), \quad u\mapsto u|_{\Omega_1}.
\end{split}
\end{equation}
 
The travel of acoustic waves through the room is for each constant wave number $k\in\R$ governed by the homogeneous Helmholtz equation
\begin{alignat}{2}\label{eq}
&\Delta u + k^2 u  = \phi &&\text{in $\Omega\setminus\bigcup_{i=1}^3S_i$},\\
\frac{\partial u}{\partial n} & = ik u \quad \text{on \,$\partial \Omega$,} \qquad
&&\frac{\partial u}{\partial n}  = 0 \quad \text{on \,$\bigcup_{i=1}^3\partial S_i$,} \nonumber
\end{alignat}
where $\phi$ is implicitly extended by zero outside of $\Omega_0$. The sound-hard scatterers are modeled by the Neumann condition on the interior Lipschitz boundaries $\partial S_i$, and on the exterior boundary $\partial\Omega$ we use an approximate Sommerfeld radiation condition; $n$ are the outward normal vectors.

The linear source-to-observable map $F$ is defined by
\begin{align}
F: \Ls\to\Ld  \qquad \phi \mapsto u|_{\Omega_1},
\end{align} 
hence, we formulate the inverse aeroacoustic source problem given noisy data $y^\delta$ as
\begin{align*}
\text{Find $\phi$ such that} \quad F(\phi)=y^\delta.  
\end{align*}

We remark that while the PDE \eqref{eq} is well-posed
\cite{ColKre13}, the observation operator $M$ renders $F$ non-injective and compact, making the inverse problem ill-posed. Thus, regularization is required to stably reconstruct the unknown acoustic source. 

We shall use the bi-level regularization scheme \cite[Algorithm 1.2]{nguyen24} for the reconstruction process. In the cited work, the author proposes a novel bi-level inversion framework. The upper level-iteration, which iteratively approximates the unknown parameter -- for this article, the source term -- embeds a lower level-iteration, which solves nonlinear PDEs inexactly. A key contribution of \cite{nguyen24} is the derivation of stopping rules for upper- and lower-level that can ensure convergence. In addition, embedding any PDE solver (FD, FEM, multigrid etc.)~into the lower-level, the error analysis in \cite{nguyen24} can inform an adaptive discretization scheme via multi-grid size, \nobreakdash-mesh size or \nobreakdash-resolution \cite{Trottenberg} that are computationally effective by not fixing a single fine scale in the inversion process.

As the model \eqref{eq} is rather classical, we shall not use the lower-level to approximate $u$. Instead, Section \ref{sec:algorithm} details how to implement a FEM solver with an adaptive \underline{mesh refinement} strategy suggested in \cite[Theorem 2]{nguyen24} to speed up the reconstruction, while, crucially, maintaining the regularizing property. This, along with the numerical verification provided in Section \ref{sec:num}, forms the contribution of this article.

\section{Inversion algorithm with mesh iterative refinement}\label{sec:algorithm}
In \cite[Algorithm 1.2]{nguyen24}, the upper-level iteratively updates the source term, and requires the adjoint of the source-to-observation $F$ to be derived.
\begin{lemma}[Adjoint]
The Hilbert space adjoint $F^*:\Ld \mapsto \Ls$ is
\begin{equation}\label{adjoint}
\begin{split}
F^*v:=\Re(z)|_{\Omega_0} \qquad&\text{ with}\\
\Delta z + k^2 z  = v \qquad&\text{in $\Omega\setminus\bigcup_{i=1}^3S_i$}, \\
\frac{\partial z}{\partial n} = -ik z \quad\text{on $\partial \Omega$,}\qquad&
\frac{\partial z}{\partial n}  = 0 \quad\text{on $\bigcup_{i=1}^3\partial S_i$,} 
\end{split}
\end{equation}
where $v$ is implicitly extended by zero outside of $\Omega_1$.
\end{lemma}
\begin{proof}
For any $\phi\in \Ls$, $v\in \Ld$, we write the inner product
\begin{align*}
\left(F\phi,v\right)_{\Ld}&:=\Re\int_{\Omega_1}u|_{\Omega_1}\overline{v}\,dx=Re\int_{\Omega_0\cup\Omega_1}u\, \overline{v}\,dx\\
&=Re\left(\Delta u+k^2u,z\right)_{L^2(\Omega_1\cup\Omega_0,\C)} + Re\left(iku-\partial_n u,z\right)_{L^2(\partial \Omega,\C)}\\
&=\left(\phi,\Re(z)\right)_{L^2(\Omega_1\cup\Omega_0)} =\left(\phi,\Re(z)|_{\Omega_0}\right)_{\Ls}\quad=:\left(\phi,F^* v\right)_{\Ls}
\end{align*}
where $\overline{(\cdot)}$ denotes the complex conjugate, $u$ is any solution of \eqref{eq} on $\Omega\setminus\bigcup_{i=1}^3S_i$, and recalling that $v$ is extended by zero outside of $\Omega_1$. The proof is complete.
\end{proof}


\begin{remark}[Multi grid size]We here remark how \cite{nguyen24} applies to multi-grid and mesh refinement. Given data noise $\delta$, \cite[Proposition 5]{nguyen24} shows that for each upper-iteration $j$, the lower-iteration should be run with precision $\epsilon_j=\delta/q^j$, $q\geq 1$, terminating the upper-level at the stopping index $j^*(\delta,q)$ according to \cite[Theorem 2]{nguyen24}. In this study, the lower-level discretization/approximation error $\epsilon_j$ translates to the FEM grid size $h_j$, where for any FEM its grid size $h$ is approximately given as $h\approx\max\left(2\int_E1\d x\right)^{1/2}$, the max being taken over elements $E$ of the FEM, approximately yielding the max height of any triangle in the mesh discretization. By \cite{BrennerScott}, the approximation error of the FEM is proportional to $h$, that is, $\epsilon = Ch$ for some $C$. By this analysis, the grid size is initially kept constant $\epsilon_0:=Ch_0>0$, and must be iteratively refined as soon as $j\geq (\ln\delta - \ln Ch_0)/\ln q$; we adaptively refine the mesh in all elements corresponding to large $h$. Expressed as an algorithm, we have \label{rem:multi}
\end{remark}

\vspace{-0.75cm}

\begin{algorithm}
\label{alg:cap}
\begin{algorithmic}
\Ensure $\phi^0:=0$, \text{step size }$\mu$, \text{threshold }$\tau$
\While{$\|y^\delta-F\phi^j\|_2 \leq \tau\delta$}
\If{$j \geq (\ln\delta - \ln Ch_0)/\ln q$}
    \State refine mesh as in Remark \ref{rem:multi}
\EndIf
\State $v\gets F_{h(\delta,j)}\phi^j - y^\delta$
\State $z \gets$ solves \eqref{adjoint} \text{ with } $F^*_{h(\delta,j)}$, $v$
\State $\phi^{j+1} \gets \phi^j - \mu z$
\EndWhile
\end{algorithmic}
\end{algorithm}

\section{Numerical experiments}\label{sec:num}

In this Section, we numerically verify our analysis by reconstruction of the source 
\vspace{-0.05cm}
\[
\phi(x_1,x_2):=
\sqrt{\min\{
\tfrac{1}{4}-x_1^2-x_2^2, 0
\}}
\cos\left(
2\pi\sqrt{x_1^2+x_2^2}
\right).
\]
\vspace{-0.35cm}

For the approximate PDE solver, we employ the NGSolve \cite{Schoberl} finite element method (FEM) package to discretize our domain $\Omega:=[-1,1]^2\subset\R^2$ and subsequently also the forward operator $F$, where $\Omega_0:=[-1/2,1/2]^2\subset\Omega$ and $\Omega_1:=[-1,1]^2\setminus [-11/20,11/20]^2\subset\R^2$ (see Figures; sound-hard scatterers $S_i$, $i=1,2,3$ are as modeled). We approximate the true state $u:=F\phi$ by applying a finely discretized FEM solver ($h^\dagger\approx 0.046$), and conduct two numerical experiments, comparing the recovery of the source term $\phi$ from noisy observations $y^\delta$ of the true state $u$, where $y^\delta-u$ consists of additive Gaussian white noise of size $1\%$ resp.~$10\%$ relative to $u$. In both experiments, two methods of recovering $\phi$ are compared.

The first method is bi-level Landweber iteration with mesh refinement as proposed above, beginning with a coarse mesh ($h_0\approx 0.531$), $q:=2^{1/60}$ and $C:=1.4/\delta$. The second method is \enquote{direct} Landweber regularization, carried out 
on a single, distinct FEM-discretized mesh that is finer than any used by the bi-level and coarser than the FEM used to generate the approximate state $u$ ($h_{\text{direct}}\approx 0.064$). In all cases, $\mu$ was estimated as $\mu\approx 0.075$, and all iteration was stopped at the first (upper-)index $j$ satisfying the discrepancy principle $\|y^\delta-F\phi^j\|_2 \leq \tau\delta$ with $\tau:=1.3$.

\begin{figure}[h!]
\centering
\includegraphics[trim={19.8cm 5cm 19.8cm 7.1cm},clip,width=0.245\textwidth,keepaspectratio]{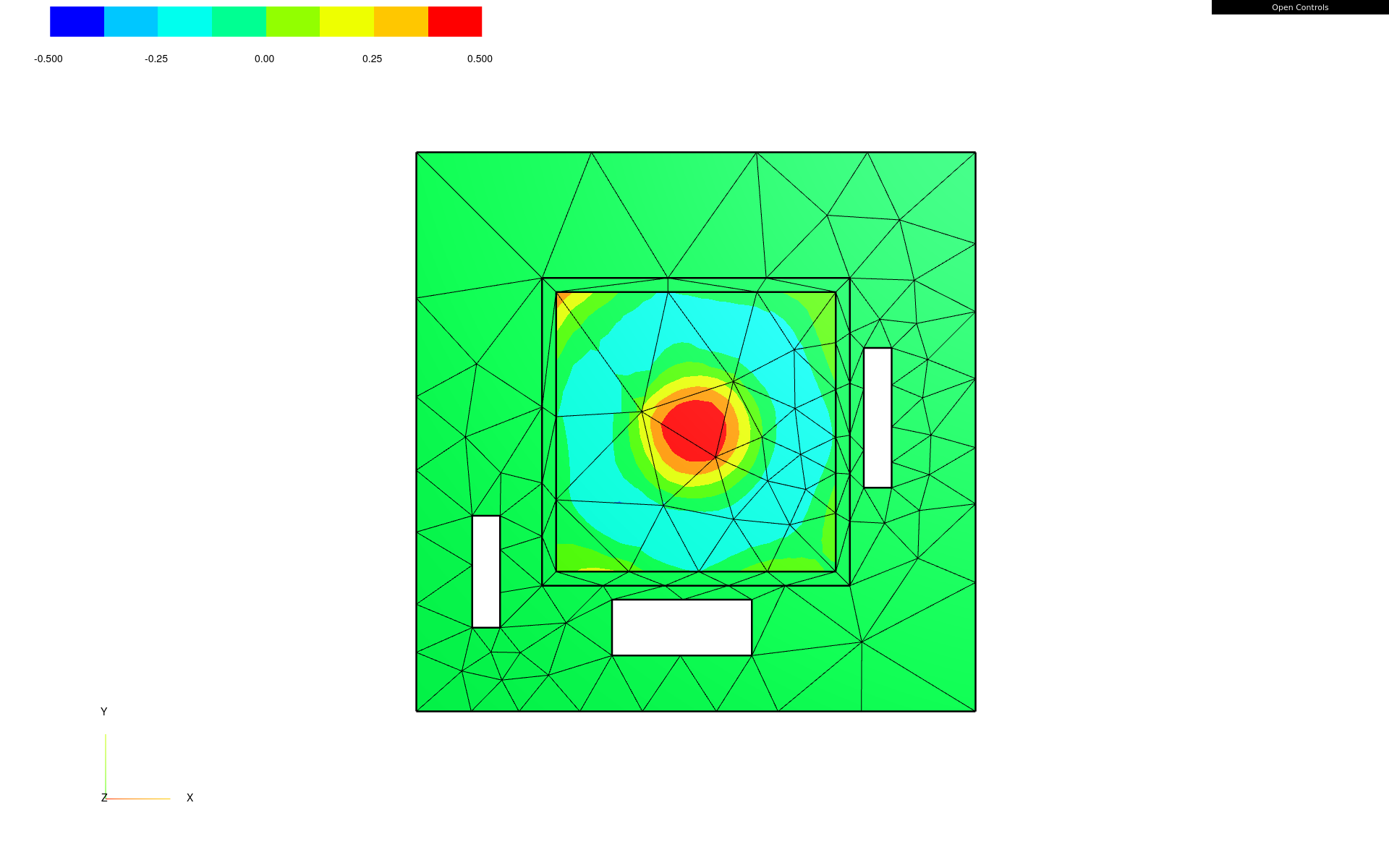}
\includegraphics[trim={19.8cm 5cm 19.8cm 7.1cm},clip,width=0.245\textwidth,keepaspectratio]{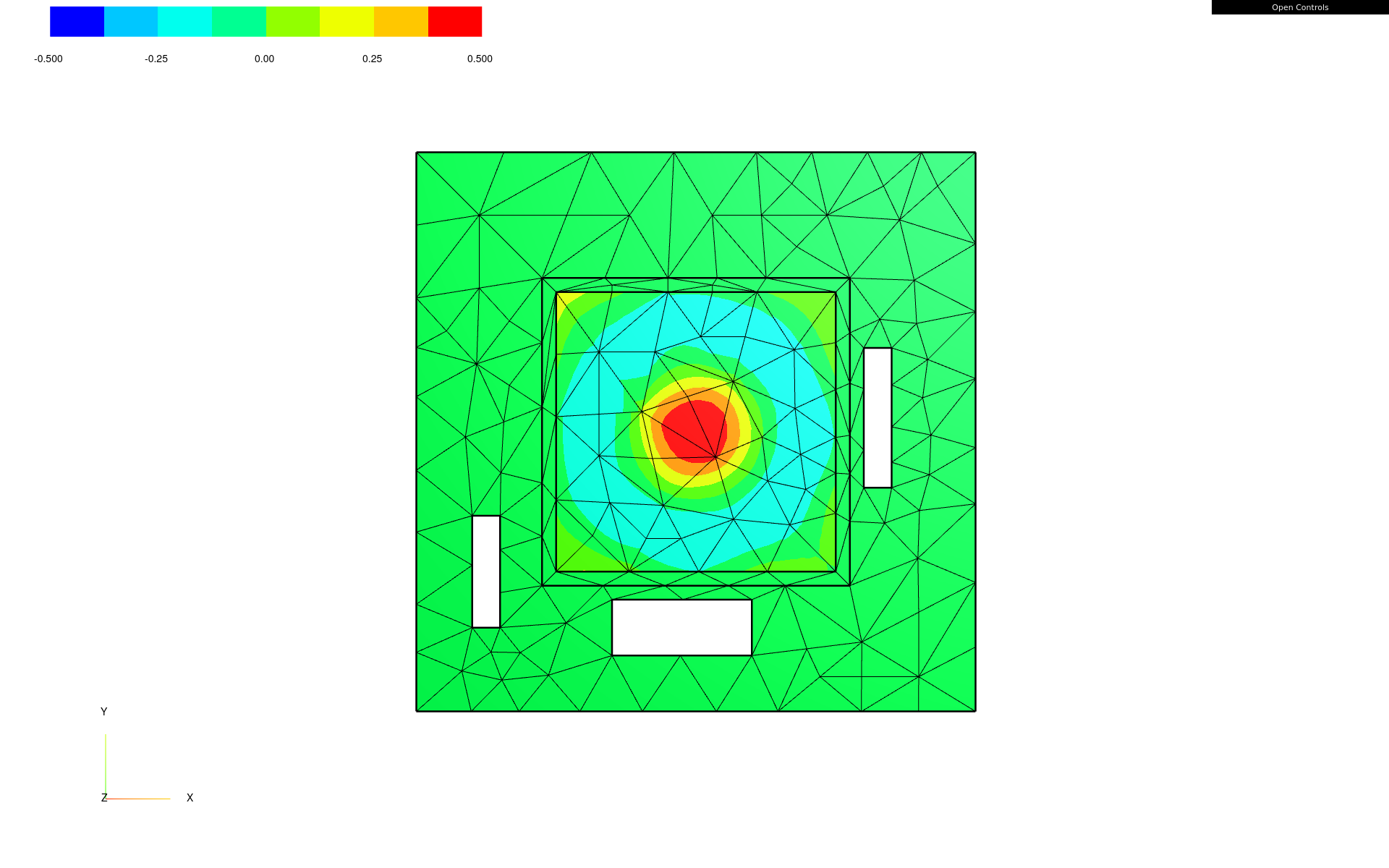} 
\includegraphics[trim={19.8cm 5cm 19.8cm 7.1cm},clip,width=0.245\textwidth,keepaspectratio]{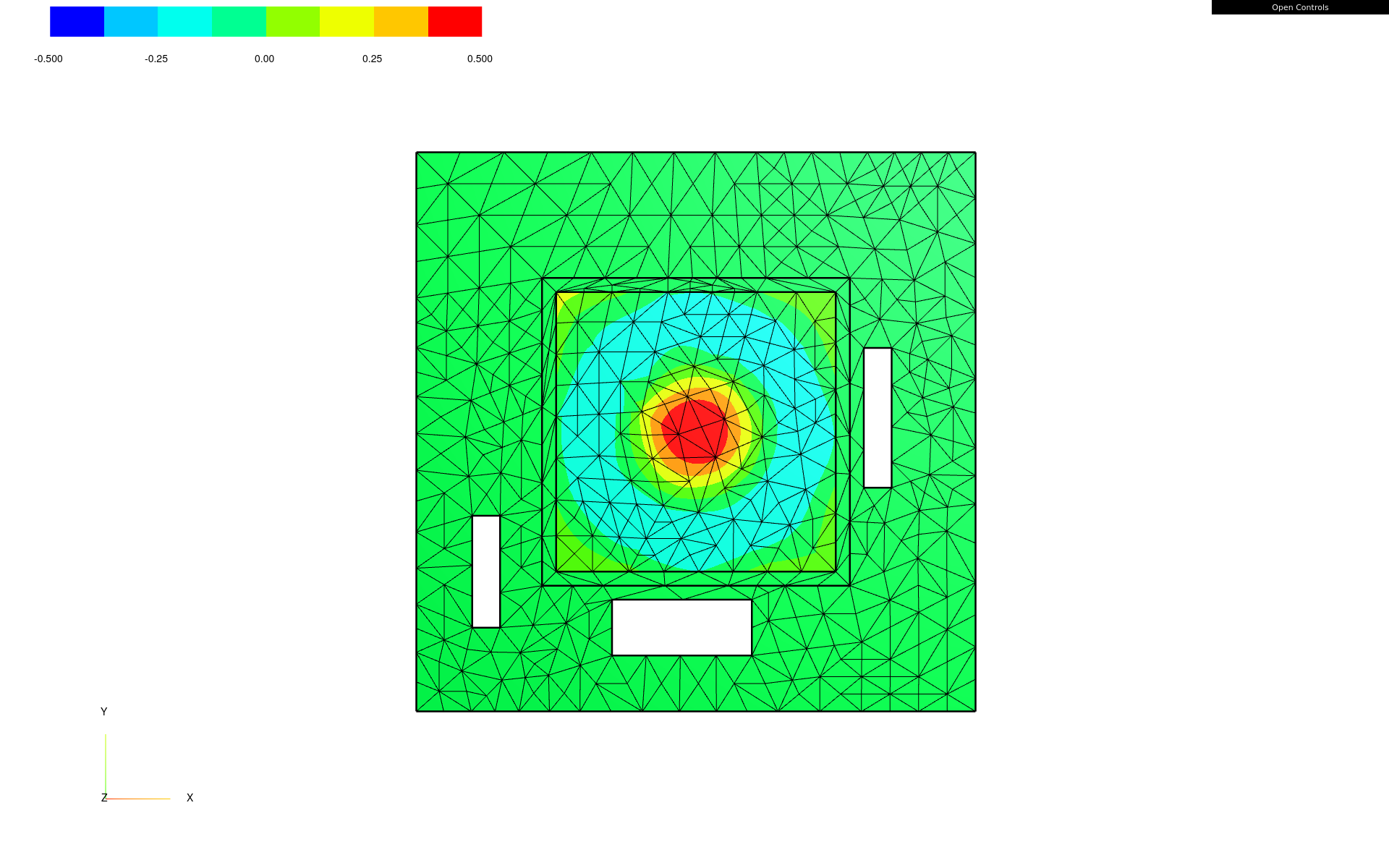}
\includegraphics[trim={19.8cm 5cm 19.8cm 7.1cm},clip,width=0.245\textwidth,keepaspectratio]{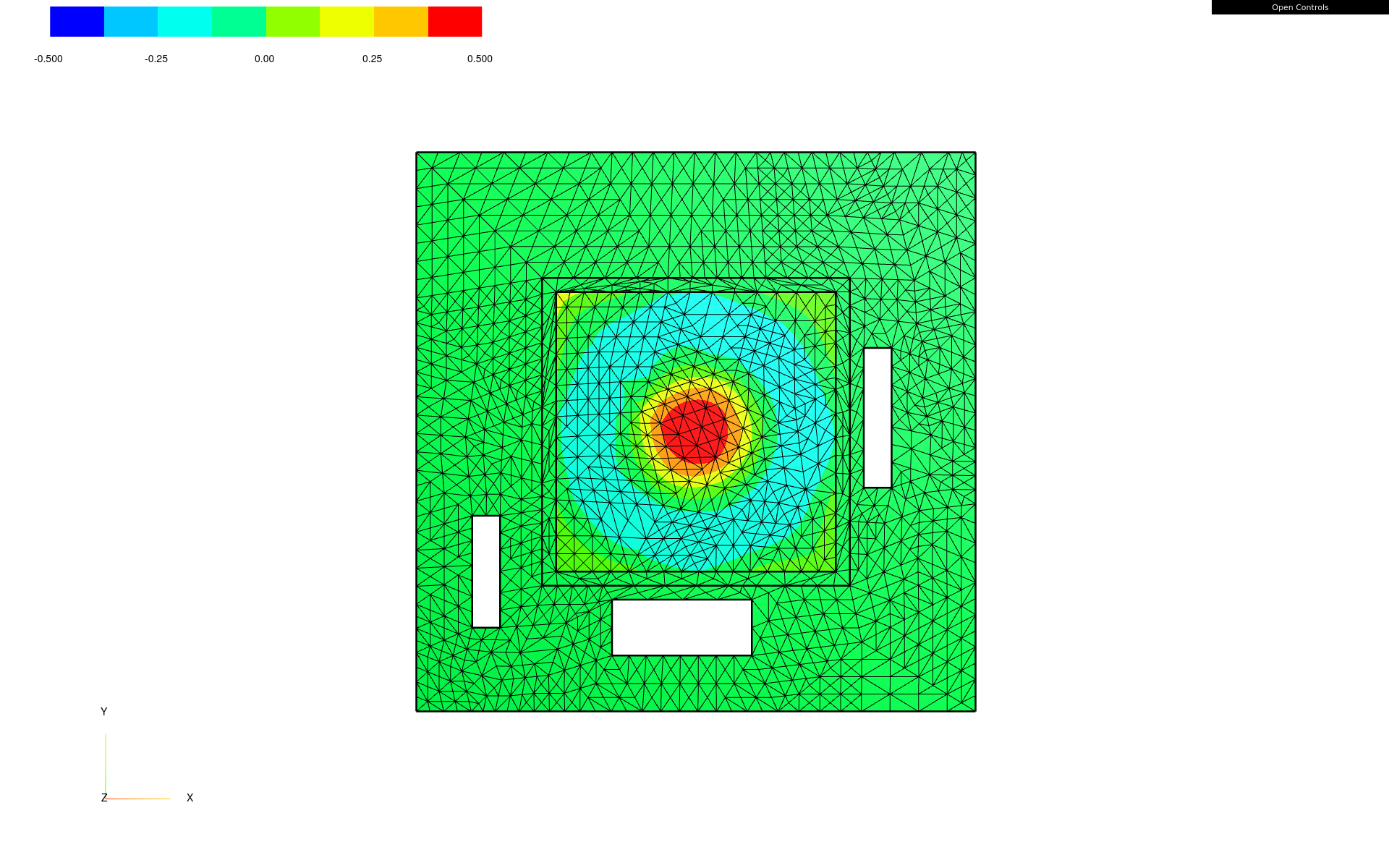}
\caption{Iterative reconstruction on each mesh in the bi-level algorithm, prior to refinement. Noise level $1\%$. $h\approx 0.531$, $0.265$, $0.139$, $0.08$.}
\end{figure}

\begin{figure}[h!]
\begin{minipage}[t]{0.49\linewidth}\vspace{0pt}
\centering
\includegraphics[trim={19.8cm 5cm 19.8cm 7.1cm},clip,width=0.49\textwidth,keepaspectratio]{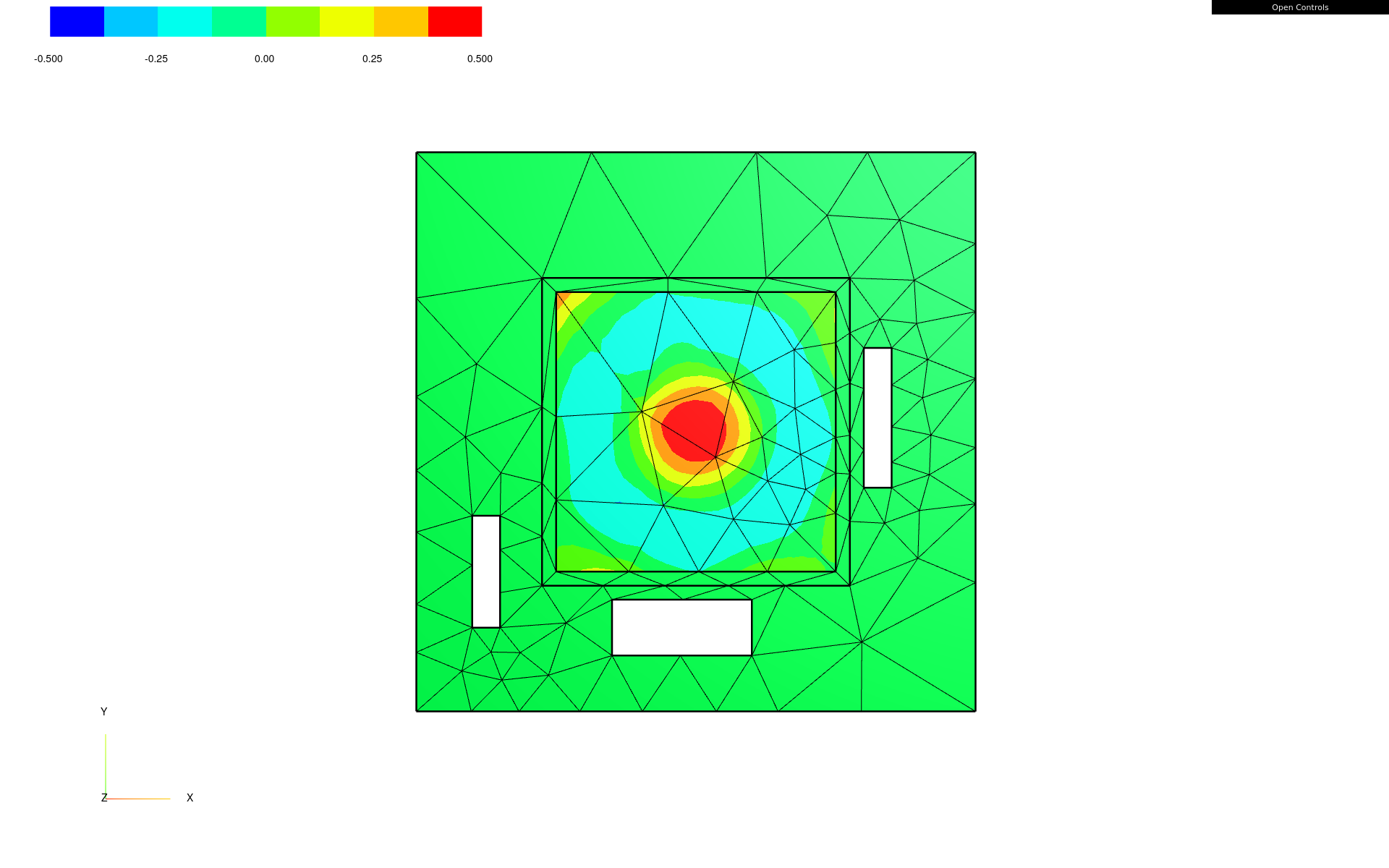}
\includegraphics[trim={19.8cm 5cm 19.8cm 7.1cm},clip,width=0.49\textwidth,keepaspectratio]{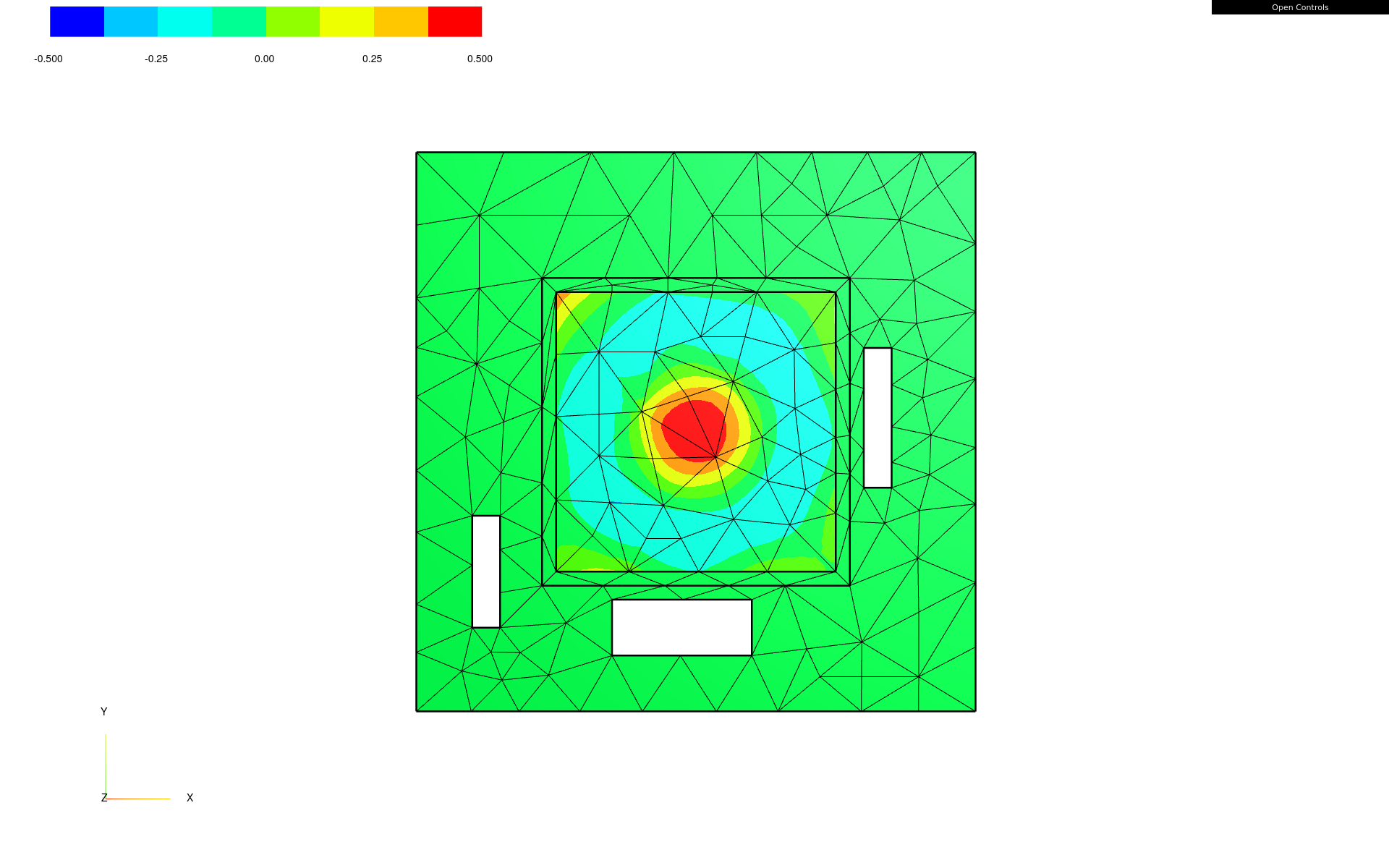}
\caption{Iterative reconstruction on each mesh in the bi-level algorithm, prior to refinement. Noise level $10\%$. $h\approx 0.531$, $0.265$.}
\end{minipage}
\hfill
\begin{minipage}[t]{0.49\linewidth}\vspace{0pt}
\centering
\includegraphics[trim={19.8cm 5cm 19.8cm 7.1cm},clip,width=0.49\textwidth,keepaspectratio]{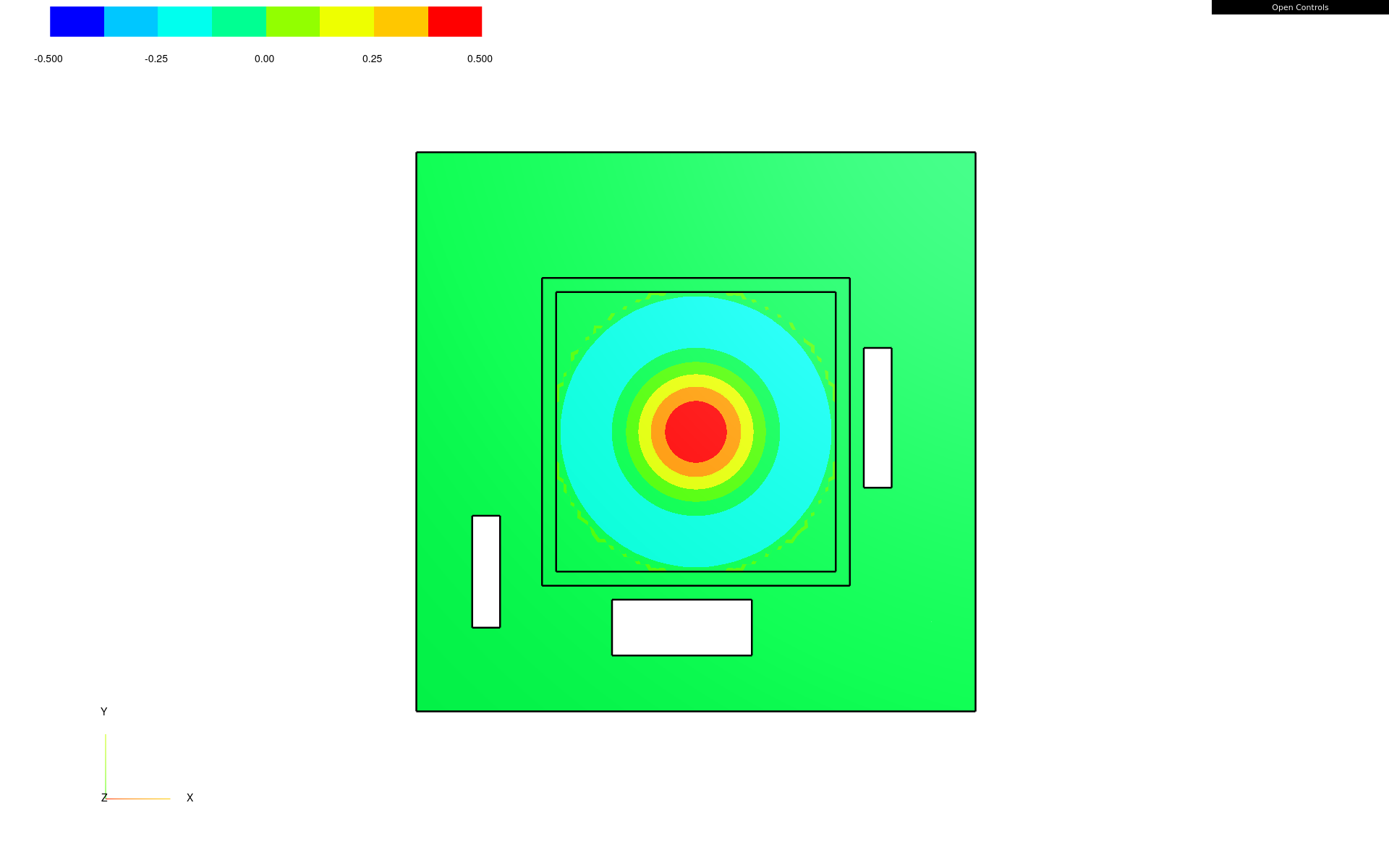}
\caption{True source $\phi$. Mesh not depicted.}
\end{minipage}
\vspace{-0.2cm}
\end{figure}

We make some remarks regarding the reconstruction. Figure \ref{fig:rec_errs} compares final reconstruction errors, while Figure \ref{fig:errs} displays reconstruction errors and residuals for the two methods, plotted by total time taken, accounting for mesh refinement (bi-level only), the Landweber step and computation of the residual.

For both $1\%$ and $10\%$ relative noise, the bi-level algorithm reached the discrepancy principle much earlier than the direct Landweber algorithm. In both cases, the final residual was smaller for the bi-level algorithm than for the direct Landweber algorithm. For $1\%$ relative noise, the final reconstruction for the direct Landweber method was slightly better than that of the bi-level algorithm, but the latter reached a higher precision much earlier despite undergoing three mesh refinements, and was only overtaken after a longer period of time. For $10\%$ relative error, the bi-level algorithm outperformed the direct Landweber method in all respects despite needing only one mesh refinement. These observations seem to justify the intuition that higher noise levels can efficiently be handled with coarser grids, an effect that the bi-level method makes explicit.


\begin{figure}
\vspace{0.1cm}
\centering
$\vcenter{\hbox{\vspace{6pt}\includegraphics[width=0.0333\textwidth,keepaspectratio]{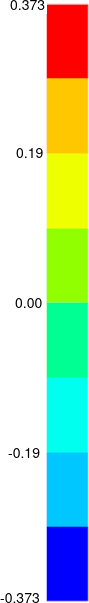}}}$
$\vcenter{\hbox{\includegraphics[trim={19.8cm 5cm 19.8cm 7.1cm},clip,width=0.23\textwidth,keepaspectratio]{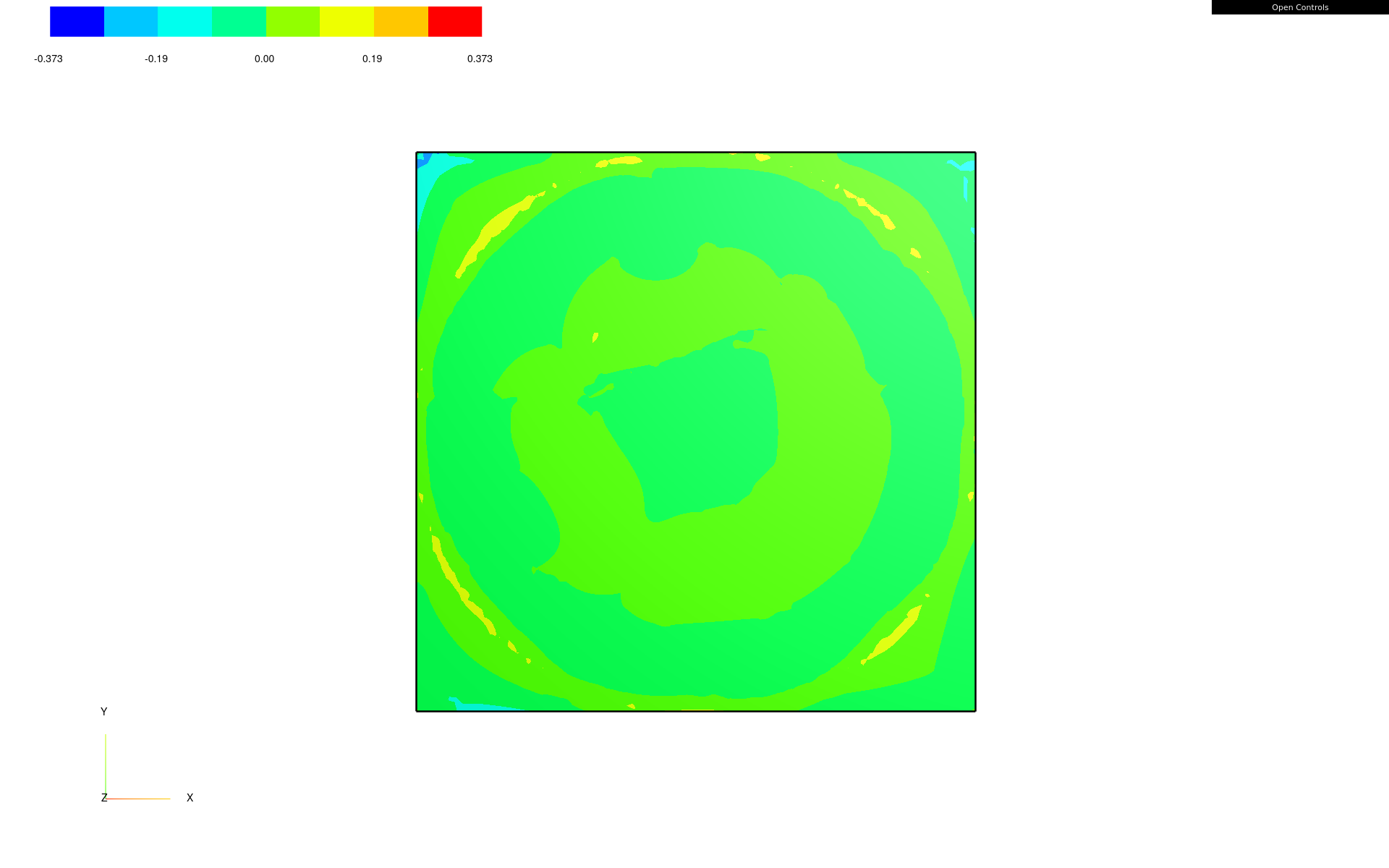}}}$
$\vcenter{\hbox{\includegraphics[trim={19.8cm 5cm 19.8cm 7.1cm},clip,width=0.23\textwidth,keepaspectratio]{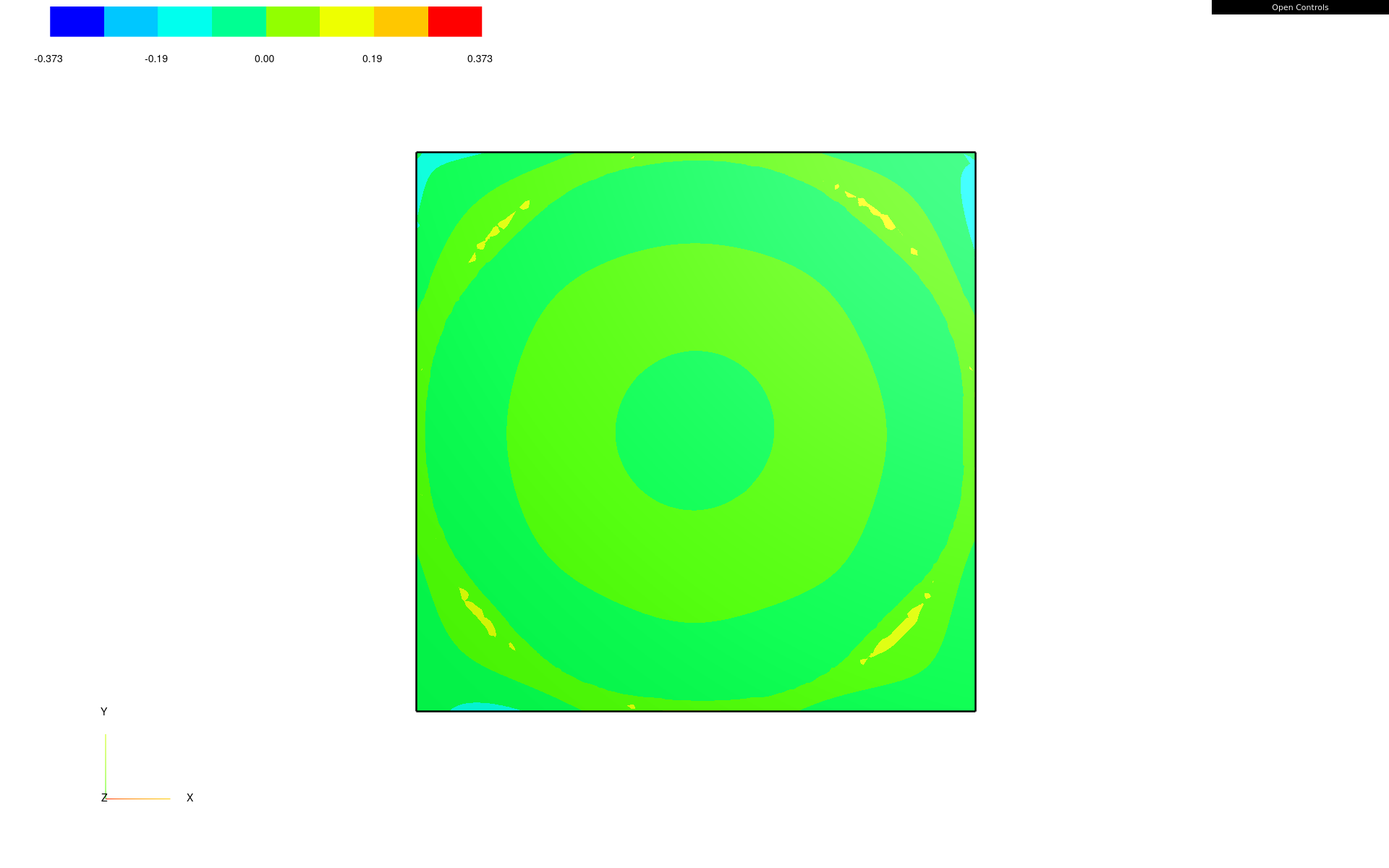}}}$
$\vcenter{\hbox{\includegraphics[trim={19.8cm 5cm 19.8cm 7.1cm},clip,width=0.23\textwidth,keepaspectratio]{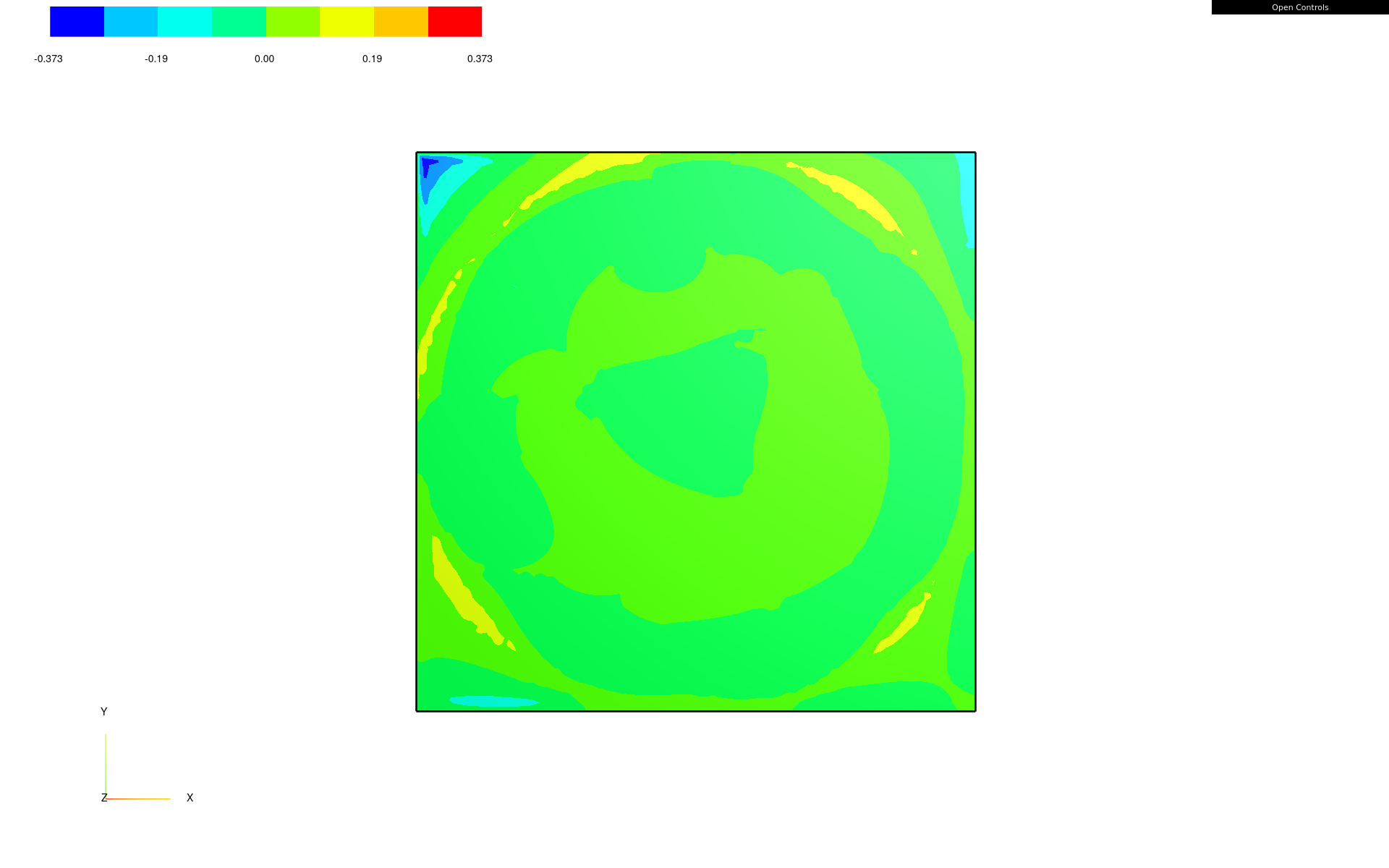}}}$
$\vcenter{\hbox{\includegraphics[trim={19.8cm 5cm 19.8cm 7.1cm},clip,width=0.23\textwidth,keepaspectratio]{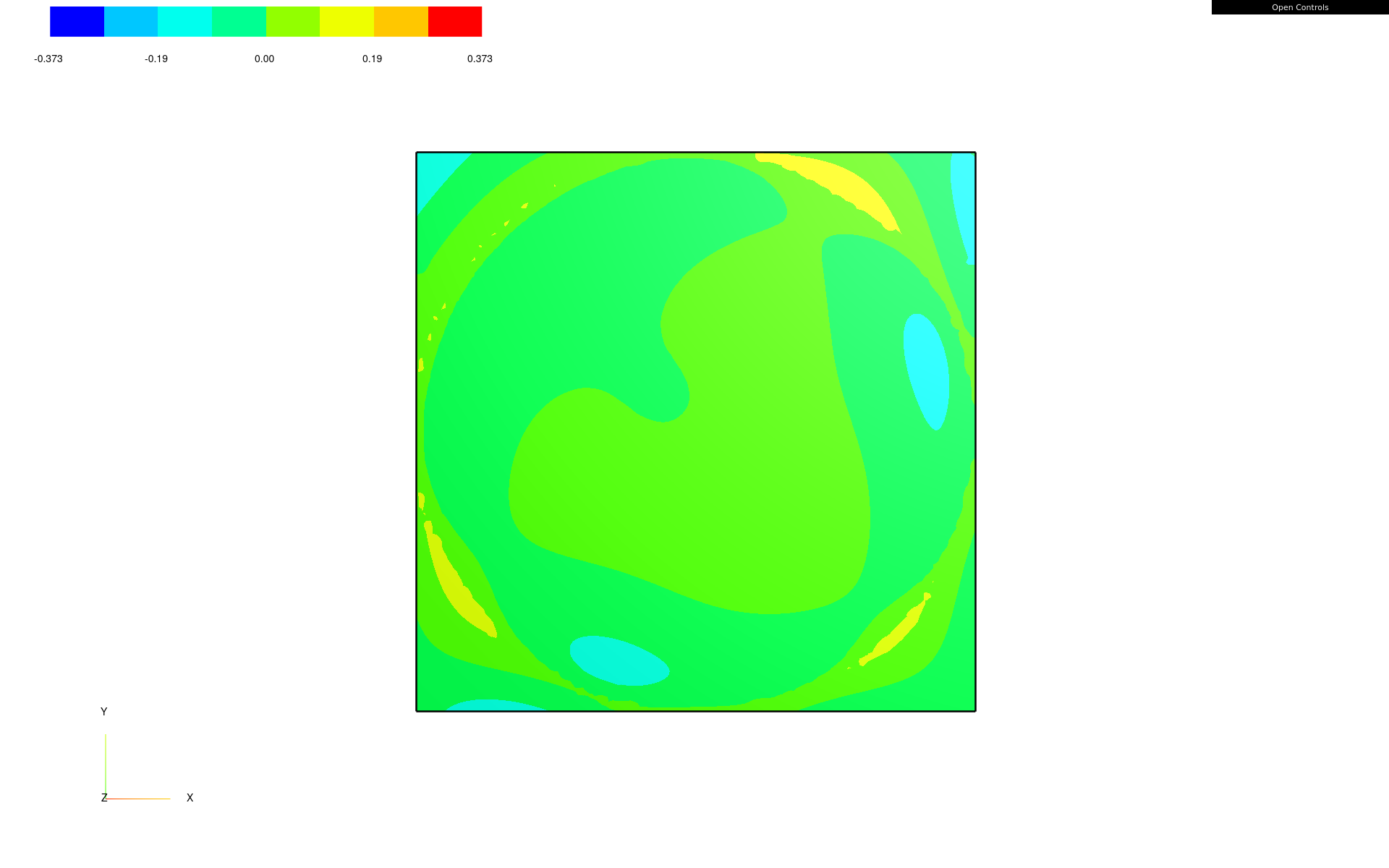}}}$
\vspace{0.1cm}
\caption{Error plots of reconstructions. Left to right: Bi-level with $1\%$ noise, direct Landweber with $1\%$ noise, bi-level with $10\%$ noise, direct Landweber with $10\%$ noise. Mesh not depicted.}
\label{fig:rec_errs}
\vspace{-0.35cm}
\end{figure}


\begin{figure}[h!]
\includegraphics[clip,width=0.49\textwidth,keepaspectratio]{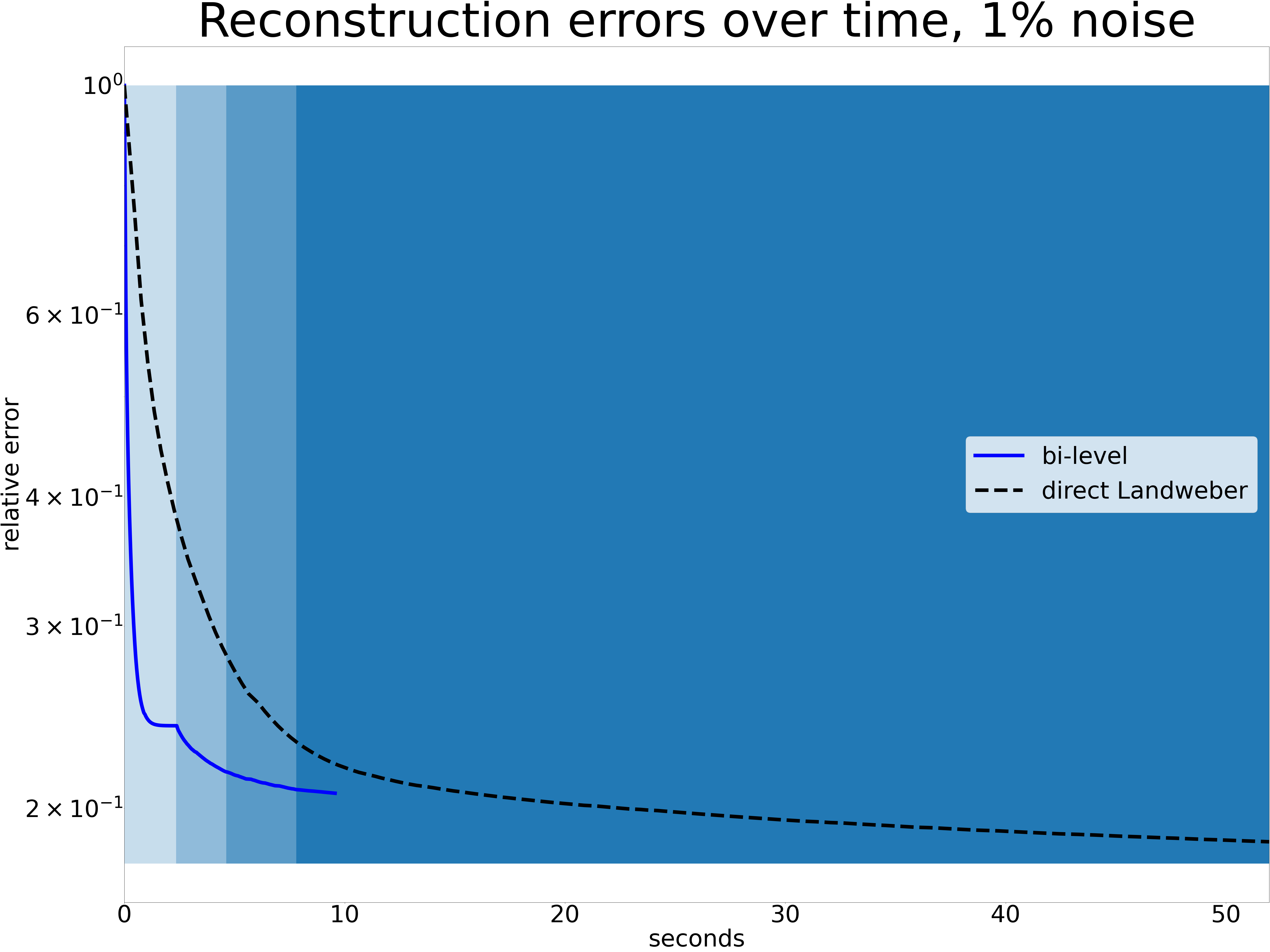}
\includegraphics[clip,width=0.49\textwidth,keepaspectratio]{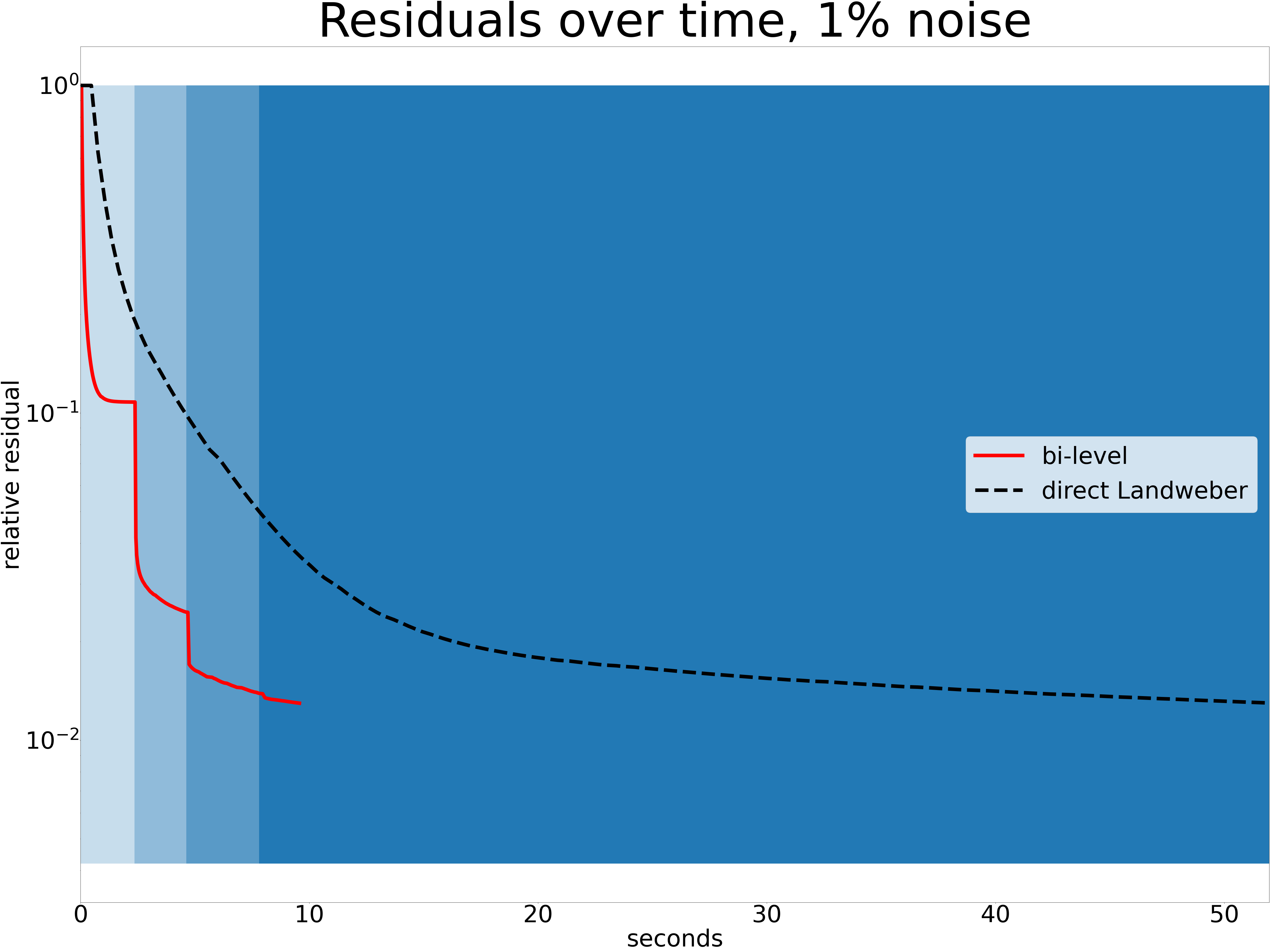} \\[2ex]
\includegraphics[clip,width=0.49\textwidth,keepaspectratio]{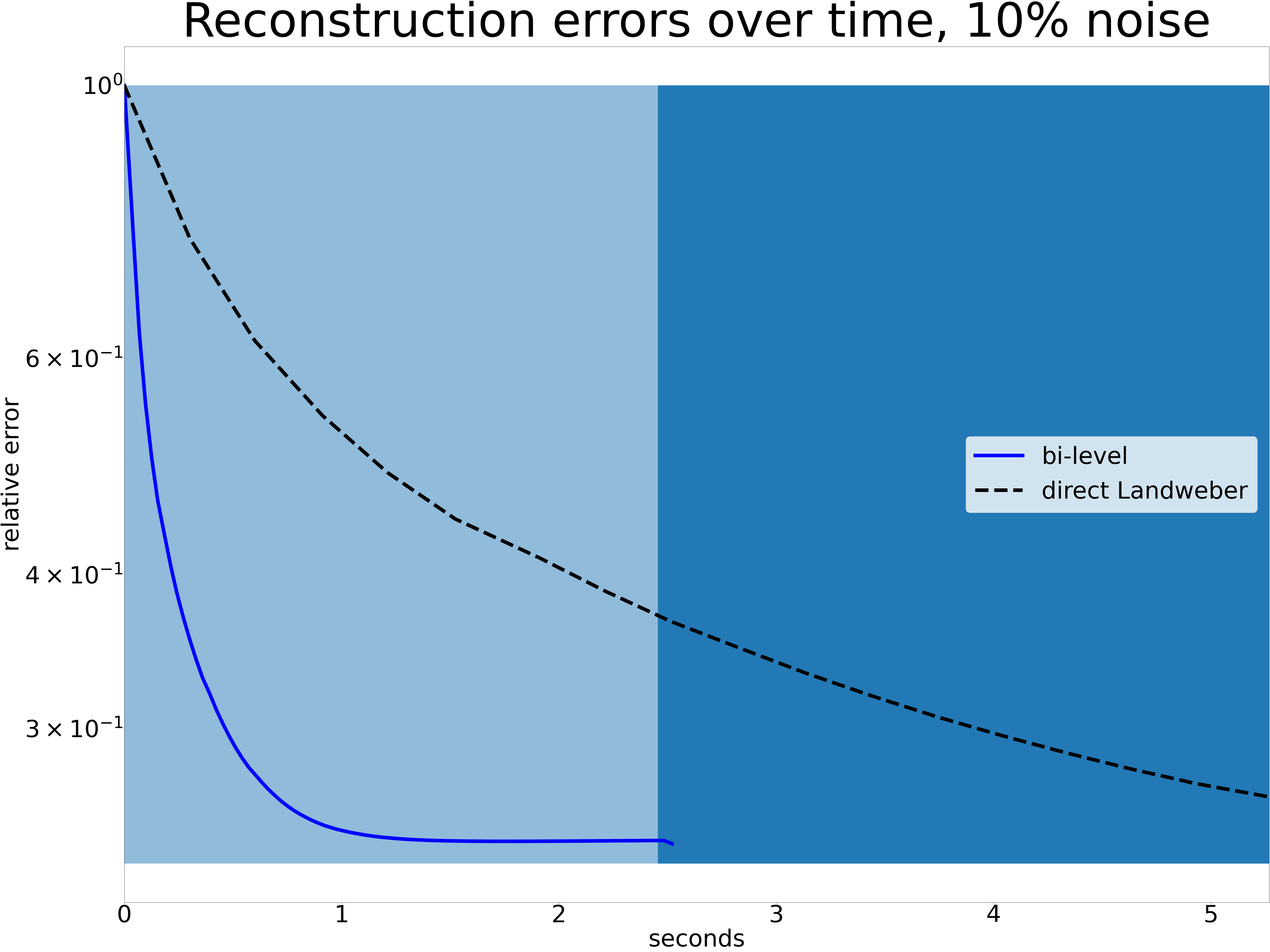}
\includegraphics[clip,width=0.49\textwidth,keepaspectratio]{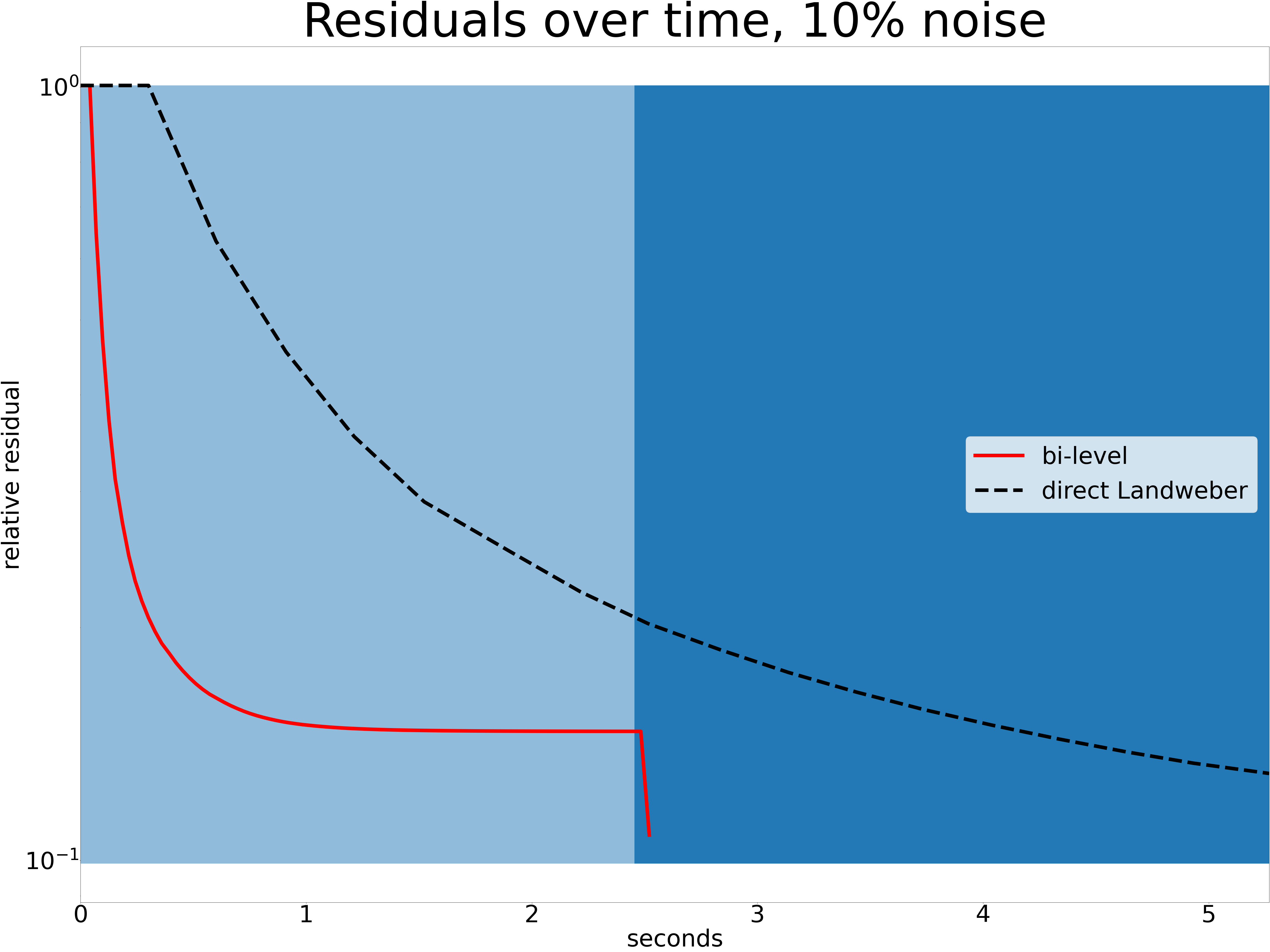}
\caption{Error plots over time. Blue background denotes which mesh refinement the bi-level algorithm is currently on.}
\vspace{-0.15cm}
\label{fig:errs}
\end{figure}

\section{Conclusions}

In this article, we have demonstrated that the bi-level algorithm produces highly efficient numerical results in a manner consistent with that claimed in \cite{nguyen24}.

The recent developments in bi-level algorithms \cite{nguyen24,nguyen-seqbi} opens up exciting usage within the field of optimal experimental design (OED). The article \cite{Aarset} newly proposed a novel method of OED for linear inverse problems. As remarked in \cite{Ucinski}, detailed treatment of the non-linear case would require additional developments within bi-level algorithms, which the authors believe may now be available.

\begin{acknowledgement}
The authors acknowledge support from the DFG through Grant 432680300 - SFB 1456 (C04). The authors moreover thank Thorsten Hohage for his comments on unique identifiability.
\end{acknowledgement}




\end{document}